\documentclass[a4paper, 12pt]{amsart}

\usepackage{amsmath}
\usepackage{amssymb}
\usepackage{ascmac}
\usepackage{amsthm}

\makeatletter

\@addtoreset{equation}{section}
\makeatother 

\theoremstyle{plain}
\newtheorem{thm}{Theorem}[section]
\newtheorem*{thm*}{Theorem}
\newtheorem{prop}[thm]{Proposition}
\newtheorem{lem}[thm]{Lemma}

\theoremstyle{definition}

\theoremstyle{remark}
\newtheorem{rem}[thm]{Remark}

\renewcommand{\epsilon}{\varepsilon}

\newcommand{\ric}{\operatorname{Ric}}

\newcommand{\bm}{\partial M}

\newcommand{\second}{{\rm I\hspace{-.01em}I}}

\usepackage{latexsym}

\title[Hamilton type entropy formula along the Ricci flow]{Hamilton type entropy formula along the Ricci flow\\ on surfaces with boundary}

\author{Keita Kunikawa}
\address{Cooperative Faculty of Education, Utsunomiya University, 350 Mine-Machi, Utsunomiya, 321-8505, Japan}
\email{kunikawa@cc.utsunomiya-u.ac.jp}

\author{Yohei Sakurai}
\address{Department of Mathematics, Saitama University, 255 Shimo-Okubo, Sakura-ku, Saitama-City, Saitama, 338-8570, Japan}
\email{ysakurai@rimath.saitama-u.ac.jp}

\subjclass[2010]{Primary 53C44; Secondly 58J32}
\keywords{Hamilton type monotonicity; Entropy; Ricci flow; Manifold with boundary}
\date{May 1, 2021}

\begin{document}
\maketitle

\begin{abstract}
In this article,
we establish a monotonicity formula of Hamilton type entropy along Ricci flow on compact surfaces with boundary.
We also study the relation between our entropy functional and the $\mathcal{W}$-functional of Perelman type.
\end{abstract}

\section{Introduction}

The aim of this short note is to formulate a monotonicity of Hamilton type entropy along Ricci flow on compact surfaces with boundary.
We further aim to investigate the relation between our entropy functional and the so-called $\mathcal{W}$-functional of Perelman type.

\subsection{Hamilton monotonicity}\label{sec:Ancient super Ricci flow}

Let $(M,g(t))_{t\in [0,T)}$ be a manifold equipped with a time-dependent Riemannian metric.
Hamilton \cite{H1} has introduced the notion of \textit{Ricci flow}
\begin{equation*}\label{eq:RF}
\partial_{t}g= -2 \ric.
\end{equation*}
He \cite{H2} has studied the (normalized) Ricci flow on closed surfaces,
and established some convergence results.
Also,
he has obtained a monotonicity of a certain entropy functional.
We consider a Ricci flow $(S^2,g(t))_{t\in [0,T)}$ on the two dimensional sphere whose initial metric has positive scalar curvature $R>0$, and volume $8 \pi$.
Notice that
the positivity of $R$ is preserved,
and the volume $v(S^2)$ evolves by $8\pi(1-t)$ (see e.g., \cite[Lemma 4.4]{B4}).
The \textit{Hamilton entropy functional} is defined by
\begin{equation}\label{eq:entropy without bdry}
\mathcal{E}(t):=\int_{S^2}\, R\log R \,dv+8\pi \log (1-t),
\end{equation}
which is non-negative (see e.g., \cite[Lemma 4.5]{B4}).
This type of entropy functional appears in the context of not only the Ricci flow theory but also minimal surface theory and Gauss curvature flow theory (see \cite[Subsection 2.3]{BM}, \cite{GPT}).

The monotonicity in \cite{H2} can be stated as follows (see \cite[Theorem 7.2]{H2}, and cf. \cite[Proposition 4.7]{B4}, \cite[Lemma 2.1]{Ch}, \cite[Exercise 9.9]{CLN}):
\begin{thm}[\cite{H2}]\label{thm:Hamilton monotonicity}
In the above situation,
let $f$ be a smooth function on $S^2$ solving
\begin{equation*}
R+\Delta f=\frac{1}{1-t}.
\end{equation*}
Then we have
\begin{equation*}
\frac{d}{dt}\,\mathcal{E}(t)=-\int_{S^2}\,\left(R\,\Vert \nabla f- \nabla \log R\Vert^2+2\left\Vert \frac{1}{2}\,R\,g+\nabla^2 f-\frac{1}{2(1-t)}g\right\Vert^2 \right)\,dv \leq 0.
\end{equation*}
In particular,
$\mathcal{E}(t)$ is non-increasing.
\end{thm}

\subsection{Hamilton type monotonicity on surfaces with boundary}\label{sec:Main theorem}

One of the purposes of this note is to generalize Theorem \ref{thm:Hamilton monotonicity} for Ricci flow on compact surfaces with boundary.
Ricci flow on manifolds with boundary has been investigated by several authors,
which is not as much as that on manifolds without boundary (see e.g., \cite{G2}, \cite{Pu1}, \cite{Pu2}, \cite{S} for short time existence and uniqueness, \cite{B1}, \cite{B2}, \cite{B3}, \cite{C}, \cite{CM}, \cite{G1}, \cite{S} for convergence, \cite{CM}, \cite{Pu2} for entropy formulas of Perelman type).
Gianniotis \cite{G2} has established a quite general short time existence and uniqueness result (see \cite[Theorem 1.2]{G2}).
As summarized in \cite{G1},
such a result holds for a given initial metric, mean curvature and induced metric on the boundary satisfying compatibility conditions.
We keep in mind the setting of \cite{G2}. 

Let $(M,g(t))_{t\in [0,T)}$ be a Ricci flow on compact surface with boundary whose initial metric has positive scalar curvature $R>0$.
We further assume a Neumann type boundary condition
\begin{equation}\label{eq:Neumann}
R_{\nu}=0
\end{equation}
for all $t\in [0,T)$,
here the left hand side means the derivative of the scalar curvature $R$ in the direction of the outward unit normal vector $\nu$ on the boundary $\partial M$.
The condition \eqref{eq:Neumann} together with the evolution formula of $R$, strong maximum principle and parabolic Hopf lemma ensures that
the positivity of $R$ is preserved (cf. \eqref{eq:evolution} below, and \cite[Section 3]{BCP}).

\begin{rem}
Li \cite{L} investigated the (normalized) Ricci flow on compact surfaces with boundary under the Neumann type boundary condition \eqref{eq:Neumann}, and obtained a short time existence result (see \cite[Corollary 6]{L}).
\end{rem}

We introduce an entropy functional of Hamilton type as follows:
\begin{equation}\label{eq:entropy with bdry}
\mathcal{E}_{\partial}(t):=\int_{M}\, R\log R \,dv-\log \overline{R} \,\int_{M} R dv,
\end{equation}
where $\overline{R}$ denotes the average of $R$ defined as 
\begin{equation*}
\overline{R}:=\frac{1}{v(M)}\int_{M} \,R\,dv.
\end{equation*}
We will verify that
$\mathcal{E}_{\partial}(t)$ is non-negative (see Proposition \ref{prop:nonnegative} below).

\begin{rem}
In virtue of the Gauss-Bonnet theorem,
the entropy functional \eqref{eq:entropy without bdry} also can be written in the form of \eqref{eq:entropy with bdry}.
Similarly,
for surfaces with boundary,
if the volume with respect to the initial metric is $4\pi \chi(M)$,
then \eqref{eq:entropy with bdry} can be written as
\begin{align*}
\mathcal{E}_{\partial}(t)&=\int_{M}\, R\log R \,dv\\
&\quad +4\pi\chi(M)\left( 1-\frac{1}{2\pi \chi(M)}\int_{\partial M}\,\kappa\,ds   \right)\log \left(  \frac{(1-t)+\dfrac{1}{2\pi \chi(M)} \displaystyle\int^{t}_{0} \int_{\partial M}\kappa\, ds\, d\xi    }{1-\dfrac{1}{2\pi \chi(M)}  \displaystyle\int_{\partial M}\,\kappa\,ds   } \right)
\end{align*}
for the Euler characteristic $\chi(M)$ of $M$, and the geodesic curvature $\kappa$ of $\partial M$.
\end{rem}

We are now in a position to state one of our main results.
\begin{thm}\label{thm:bdry Hamilton monotonicity}
Let $(M,g(t))_{t\in [0,T)}$ be a Ricci flow on compact surface with boundary whose initial metric has positive scalar curvature.
We further assume the Neumann type boundary condition \eqref{eq:Neumann}.
Let $f$ be a smooth function on $M$ solving the Nuemann boundary problem
\begin{equation}\label{eq:test}
\begin{cases}
                                                    \displaystyle R+\Delta f= \overline{R} & \text{in $M$}, \\
                                                                                f_{\nu}=0           & \text{on $\partial M$}.
                                                   \end{cases}
\end{equation}
Then we have
\begin{align}\label{eq:main}
\frac{d}{dt}\,\mathcal{E}_{\partial}(t)&=-\int_{M}\,\left(R\,\Vert \nabla f-\nabla \log R \Vert^2+2\left\Vert \frac{1}{2}\,R\,g+\nabla^2f-\frac{1}{2}\,\overline{R}\,g \right\Vert^2\right)\,dv\\ \notag
                                                        &\quad -2\int_{\partial M} \kappa\, \Vert \nabla_{\partial M} (f|_{\partial M}) \Vert^2 ds.
\end{align}
In particular,
if $\partial M$ is convex with respect to the initial metric $($i.e., $\kappa(0)$ is non-negative$)$, then $\mathcal{E}_{\partial}(t)$ is non-increasing.
\end{thm}

\begin{rem}
The convexity of $\partial M$ is preserved in our setting.
Actually,
we possess the following evolution formula of $\kappa$ along Ricci flow (see e.g., \cite[(3.22)]{S}, \cite[Proposition 2.1]{CM}):
\begin{equation*}
\partial_t \kappa=\frac{1}{2}\kappa R-\frac{1}{2}R_{\nu}.
\end{equation*}
In particular,
under the condition \eqref{eq:Neumann},
we can solve it as follows:
\begin{equation*}
\kappa(t)=\kappa(0)\exp \left(\frac{1}{2} \int^{t}_{0}\,R\, d\xi  \right).
\end{equation*}
This tells us the desired claim.
Once the boundary $\partial M$ becomes convex,
the Gauss-Bonnet theorem and the positivity of $R$ yield that $\chi(M)$ must be positive.
\end{rem}

\begin{rem}\label{rem:main critical}
We mention the critical point of $\mathcal{E}_{\partial}(t)$.
Let $\partial M$ be convex.
Suppose that the time derivative of $\mathcal{E}_{\partial}(t)$ vanishes at time $t_0>0$.
Then by Theorem \ref{thm:bdry Hamilton monotonicity} we see
\begin{equation*}
\nabla f=\nabla \log R,\quad \frac{1}{2}\,R\,g+\nabla^2f=\frac{1}{2}\,\overline{R}\,g,\quad \kappa\, \Vert \nabla_{\partial M} (f|_{\partial M}) \Vert^2=0.
\end{equation*}
The first identity says that
$f$ agrees with $\log R$ up to addition by a constant.
The second one means that
it is a gradient shrinking Ricci soliton.
We now observe the third identity.
In the case where $\kappa$ is positive everywhere,
$f$ must be constant over $\partial M$;
in particular,
$R$ also has the same property due to the first identity and the Neumann boundary conditions \eqref{eq:Neumann} and \eqref{eq:test}.
On the other hand,
when $\kappa=0$,
the boundary $\partial M$ is geodesic;
in particular,
we can consider the double of $M$,
and its universal cover is a smooth gradient shrinking Ricci soliton over $S^2$.
It is well-known that every Ricci soliton on $S^2$ has constant scalar curvature (see e.g., \cite[Corollary 9.11]{CLN}).
Thus $M$ also has constant scalar curvature.
\end{rem}

\begin{rem}
Cortissoz-Murcia \cite{CM} have obtained monotonicity formulas of Perelman type in a similar setting (see \cite[Theorems 3.1 and 3.2]{CM}, and also \cite[Theorem 3.3]{Pu2}).
One can observe that a similar boundary term to that in \eqref{eq:main} appears.
Also, Ni \cite{N} has formulated monotonicity of Perelman type for linear heat equation on static manifolds with boundary (see \cite[Corollary 3.1]{N}, and also \cite[Problem 8.3]{LL}).
In the same manner,
a similar boundary term appears.
\end{rem}

\subsection{Perelman type monotonicity on surfaces with boundary}\label{sec:Main theorem2}
Motivated by the question of Ni \cite{N},
Guo \cite{Gu} has examined a relation between the Hamilton entropy functional and the so-called $\mathcal{W}$-functional introduced by Perelman \cite{P} on closed surfaces.
Inspired by \cite{Gu},
we investigate the relation between our entropy functional \eqref{eq:entropy with bdry} and the $\mathcal{W}$-functional on compact surfaces with boundary.

We consider the setting in Subsection \ref{sec:Main theorem}.
We introduce a $\mathcal{W}$-functional of Guo type by
\begin{equation*}
\mathcal{W}_{\partial}(t):=\int_{M}\,[\tau(R-\Vert \nabla \log R \Vert^2)-\log R-\log \tau] \,R\,dv-2\log \tau \,\int_{\partial M}\,\kappa\,ds,
\end{equation*}
where
\begin{equation*}
\tau:=T-t.
\end{equation*}

\begin{rem}\label{rem:relation}
By the inequality \eqref{eq:time2.2} stated below, and the Gauss-Bonnet theorem,
we possess the following relation between $\mathcal{E}_{\partial}(t)$ and $\mathcal{W}_{\partial}(t)$:
\begin{equation*}
\mathcal{W}_{\partial}(t)=\tau\, \left(\frac{d}{dt}\,\mathcal{E}_{\partial}(t)\right)-\mathcal{E}_{\partial}(t)-4\pi \chi(M)\log \tau+\tau \,v(M)\,\overline{R}^2-\log \overline{R}\,\int_{M}\,R\,dv.
\end{equation*}
\end{rem}

Our second main theorem is the following (cf. \cite[Theorem 1.2]{Gu}):

\begin{thm}\label{thm:bdry Guo monotonicity}
Let $(M,g(t))_{t\in [0,T)}$ be a Ricci flow on compact surface with boundary whose initial metric has positive scalar curvature.
We further assume the Neumann type boundary condition \eqref{eq:Neumann}.
Then we have
\begin{align*}\label{eq:main2}
\frac{d}{dt}\,\mathcal{W}_{\partial}(t)&=2\tau \int_{M}\,R \left\Vert \frac{1}{2}Rg+\nabla^2 \log R-\frac{1}{2\tau}g \right\Vert^2\,dv\\ \notag
                                                        &\quad +2\tau \int_{\partial M} \kappa\, \left( R\,\Vert \nabla_{\partial M}\,(\log R)|_{\partial M} \Vert^2   +\frac{1}{\tau^2} \right) ds.
\end{align*}
In particular,
if $\partial M$ is convex with respect to the initial metric, then $\mathcal{W}_{\partial}(t)$ is non-decreasing.
\end{thm}

\begin{rem}\label{rem:main critical2}
Let us discuss the critical point of $\mathcal{W}_{\partial}(t)$.
Let $\partial M$ be convex.
Assume that its derivative vanishes at $t_0>0$.
In view of Theorem \ref{thm:bdry Guo monotonicity},
\begin{equation*}
\frac{1}{2}\,R\,g+\nabla^2 \log R=\frac{1}{2\tau}\,g,\quad \kappa=0.
\end{equation*}
It follows that
it is a gradient shrinking Ricci soliton with geodesic boundary.
From the same discussion as in Remark \ref{rem:main critical},
we can conclude that $M$ has constant scalar curvature.
\end{rem}

\section{Proof}

Let us prove our main theorems.

\subsection{Reilly formula}\label{sec:Reilly formulas}
In this subsection,
we recall the Reilly formula,
which is a key ingredient of the proof of our main theorem.
Let $(M,g)$ be a compact manifold with boundary.
The second fundamental form of $\bm$ is defined as
\begin{equation*}
\second(X,Y):=g(\nabla_{X} \nu,Y)
\end{equation*}
for tangent vectors $X,Y$ on $\bm$,
where $\nabla$ denotes the Levi-Civita connection.
Note that
in the two dimensional case,
the geodesic curvature $\kappa$ coincides with $\second(v,v)$ for a unit tangent vector $v$ of $\bm$.
Furthermore,
the mean curvature $H$ is defined as the trace of $\second$.
We possess:
\begin{thm}[\cite{R}]\label{thm:Reilly}
For all $f\in C^{\infty}(M)$, it holds that
\begin{align*}
&\quad \,\int_{M}\, \left(\Delta f \right)^{2}- \ric(\nabla f)-\left\Vert \nabla^2 f \right\Vert^{2} \,dv\\ 
&=\int_{\bm}\, 2f_{\nu}\,\Delta_{\bm} (f|_{\partial M})+f_{\nu}^{2}\,H+\second \left(\nabla_{\bm} (f|_{\bm}), \nabla_{\bm} (f|_{\bm}) \right) \,dv_{\bm}.
\end{align*}
\end{thm}

For later convenience,
we also recall the following well-known and useful formula,
which is used in the standard proof of Reilly formula (see e.g., \cite[Chapter 8]{Li}):
\begin{lem}\label{lem:useful}
For all $f\in C^{\infty}(M)$ we have
\begin{align*}
\left(\Vert \nabla f \Vert^{2}\right)_{\nu}=2\,f_{\nu} \left[\Delta f-\Delta_{\bm} (f|_{\bm})-f_{\nu} H  \right]&+2\,g_{\bm}\left(\nabla_{\bm}(f|_{\partial M}),\nabla_{\bm}f_{\nu} \right)\\
&-2\,\second \left(\nabla_{\bm}(f|_{\partial M}),\nabla_{\bm}(f|_{\partial M})  \right). \notag
\end{align*}
\end{lem}

\subsection{Proof of Hamilton type monotonicity}\label{sec:Hamilton}

Let $(M,g(t))_{t\in [0,T)}$ be as in Theorem \ref{thm:bdry Hamilton monotonicity}.
We first verify the following (cf. \cite[Lemma 4.5]{B4}):
\begin{prop}\label{prop:nonnegative}
For all $t\in [0,T)$ we have $\mathcal{E}_{\partial}(t)\geq 0$.
\end{prop}
\begin{proof}
We see $\log c\geq 1-c^{-1}$ for all $c>0$,
and hence
\begin{equation*}
\mathcal{E}_{\partial}(t)=\int_{M}\,R \log \left( \frac{R}{\overline{R}}  \right) \,dv\geq \int_{M}\,\left(R-\overline{R}\right) \,dv=0.
\end{equation*}
This proves the desired estimate.
\end{proof}

Before we state the next assertion,
we recall the following evolution formulas along Ricci flow (see e.g., \cite[Corollaries 4.16 and 4.20]{AH}):
\begin{equation}\label{eq:evolution}
\partial_t dv=-R\, dv,\quad \partial_t R=\Delta R+R^2.
\end{equation}

Let us calculate the following:
\begin{lem}\label{lem:time}
\begin{equation*}
\partial_t \left(\log \overline{R} \,\int_{M} R dv\right)=v(M)\overline{R}^2.
\end{equation*}
\end{lem}
\begin{proof}
By using \eqref{eq:evolution} and integration by parts with \eqref{eq:Neumann}, we see
\begin{align*}
\partial_t \overline{R}&=\frac{1}{v(M)^2}\left\{v(M)\left( \int_{M}\,\partial_t R\,dv-\int_{M}\,R^2\,dv\right)-\partial_t v(M) \int_{M}\,R\,dv \right\}=\overline{R}^2.
\end{align*}
It follows that
\begin{align*}
\partial_t \left(\log \overline{R} \,\int_{M} R dv\right)&=\frac{\partial_t \overline{R}}{\overline{R}}\int_{M} R dv+\log \overline{R} \left( \int_{M}\,\partial_t R\,dv-  \int_{M}\,R^2\,dv\right)\\
             &=\frac{\partial_t \overline{R}}{\overline{R}}\int_{M} R dv=v(M)\overline{R}^2.
\end{align*}
We arrive at the desired formula.
\end{proof}

Lemma \ref{lem:time} yields the following:
\begin{lem}\label{lem:time2}
\begin{align}\label{eq:time2.1}
\frac{d}{dt}\mathcal{E}_{\partial}(t)&=\int_{M}\, \left((\Delta R) \log R+R^2\right)\,dv-v(M)\overline{R}^2\\ \label{eq:time2.2}
&=\int_{M}\, \left(-\Vert \nabla \log R \Vert^2+R\right)R\,dv-v(M)\overline{R}^2.
\end{align}
\end{lem}
\begin{proof}
The formulas \eqref{eq:evolution}, Lemma \ref{lem:time}, and integration by parts with \eqref{eq:Neumann} imply
\begin{align*}
\frac{d}{dt}\mathcal{E}_{\partial}(t)&=\int_{M}\, \left((\partial_t R)(1+\log R)-R^2 \log R\right)\,dv-\partial_t \left(\log \overline{R} \,\int_{M} R dv\right)\\ \notag
&=\int_{M}\, \left((\Delta R+R^2)(1+\log R)-R^2 \log R\right)\,dv-v(M)\overline{R}^2  \\ \notag
&=\int_{M}\, \left((\Delta R) \log R+R^2\right)\,dv-v(M)\overline{R}^2,
\end{align*}
which is \eqref{eq:time2.1}.
The equality \eqref{eq:time2.2} follows from integration by parts with \eqref{eq:Neumann}.
\end{proof}

We now prove Theorem \ref{thm:bdry Hamilton monotonicity}:
\begin{proof}[Proof of Theorem \ref{thm:bdry Hamilton monotonicity}]
Let $(M,g(t))_{t\in [0,T)}$ and $f$ be as in Theorem \ref{thm:bdry Hamilton monotonicity}.
Using \eqref{eq:time2.2},
we have
\begin{align}\label{eq:proof1}
\frac{d}{dt}\mathcal{E}_{\partial}(t)=\int_{M}\, \left(-R\Vert \nabla \log R \Vert^2+R^2\right)\,dv-v(M)\overline{R}^2.
\end{align}
We also deduce
\begin{equation*}
R^2=(\Delta f)^2-2\overline{R}\, \Delta f+\overline{R}^2
\end{equation*}
from \eqref{eq:test}.
By integration by parts with $f_{\nu}=0$,
we have
\begin{equation}\label{eq:proof2}
\int_{M}\,R^2\,dv=\int_{M}\,(\Delta f)^2\,dv+v(M)\,\overline{R}^2.
\end{equation}
Combining \eqref{eq:proof1} and \eqref{eq:proof2},
we obtain
\begin{equation}\label{eq:proof2.5}
\frac{d}{dt}\mathcal{E}_{\partial}(t)=\int_{M}\, \left(-R\Vert \nabla \log R \Vert^2+(\Delta f)^2\right)\,dv.
\end{equation}

On the other hand,
by integration by parts with $f_{\nu}=0$, and by \eqref{eq:test},
\begin{align}\label{eq:proof3}
&\quad \,\,\int_{M}  R\Vert \nabla f\Vert^2-2(\Delta f)^2+  R\Vert \nabla \log R \Vert^2 \,dv\\ \notag
&=\int_{M}R\Vert \nabla f\Vert^2+2g(\nabla f,\nabla \Delta f)+R\Vert \nabla \log R \Vert^2 \,dv\\ \notag
&=\int_{M}\,R \Vert \nabla f-\nabla \log R\Vert^2\,dv.
\end{align}
Furthermore,
Theorem \ref{thm:Reilly} yields
\begin{equation}\label{eq:proof4}
\int_{M}\,2 \left(\Delta f \right)^{2}- R\Vert \nabla f \Vert^2-2\left\Vert \nabla^2 f \right\Vert^{2} \,dv=2\int_{\bm}\, \second \left(\nabla_{\bm} (f|_{\bm}), \nabla_{\bm} (f|_{\bm}) \right) \,dv_{\bm}.
\end{equation}
Summing up \eqref{eq:proof3} and \eqref{eq:proof4},
we see
\begin{align}\label{eq:proof5}
\int_{M}R\Vert \nabla \log R \Vert^2-2\left\Vert \nabla^2 f \right\Vert^{2} \,dv&=\int_{M}\,R \Vert \nabla f-\nabla \log R\Vert^2\,dv\\ \notag
&\quad \,\,+2\int_{\bm}\, \second \left(\nabla_{\bm} (f|_{\bm}), \nabla_{\bm} (f|_{\bm}) \right) \,dv_{\bm}.
\end{align}
Moreover,
summing up \eqref{eq:proof2.5} and \eqref{eq:proof5} implies
\begin{align*}
&\quad \,\,\frac{d}{dt}\mathcal{E}_{\partial}(t)+\int_{M}\,R \Vert \nabla f-\nabla \log R\Vert^2\,dv+2\int_{\bm}\, \second \left(\nabla_{\bm} (f|_{\bm}), \nabla_{\bm} (f|_{\bm}) \right) \,dv_{\bm}\\
&=\int_{M}\, (\Delta f)^2-2\Vert \nabla^2 f\Vert^2\,dv=-2\int_{M}\, \left\Vert \nabla^2 f-\frac{\Delta f}{2}g \right\Vert^2\,dv.
\end{align*}
We finally apply \eqref{eq:test} to the right hand side.
This completes the proof.
\end{proof}

\subsection{Proof of Guo type monotonicity}\label{sec:Guo}

Let $(M,g(t))_{t\in [0,T)}$ be as in Theorem \ref{thm:bdry Guo monotonicity}.
In order to prove Theorem \ref{thm:bdry Guo monotonicity},
we prepare several lemmas.
We start with the following:
\begin{lem}\label{lem:time normal}
\begin{equation*}
\partial_t \nu=\frac{R}{2}\nu.
\end{equation*}
\end{lem}
\begin{proof}
Let us consider a (time-independent) local coordinate $x_1$ on $\partial M$,
and take the time derivative of $g(\nu,\partial_1)=0$.
From the two dimensional Ricci flow equation and \eqref{eq:Neumann},
\begin{equation*}
0=(\partial_t g)\left(\nu,\partial_1 \right)+g(\partial_t \nu,\partial_1)=(-Rg)(\nu,\partial_1)+g(\partial_t \nu,\partial_1)=g(\partial_t \nu,\partial_1),
\end{equation*}
and hence $\partial_t \nu$ is orthogonal to $\partial M$.
Next,
we take the time derivative of $g(\nu,\nu)=1$.
Then
\begin{equation*}
0=(\partial_t g)\left(\nu,\nu \right)+2g(\partial_t \nu,\nu)=(-Rg)(\nu,\nu)+2g(\partial_t \nu,\nu)=-R+2g(\partial_t \nu,\nu).
\end{equation*}
We arrive at the desired claim.
\end{proof}

Having at hand Lemma \ref{lem:time normal},
we obtain the following:
\begin{lem}\label{lem:Laplacenorm}
\begin{equation*}
(\Delta R)_{\nu}=0.
\end{equation*}
\end{lem}
\begin{proof}
From \eqref{eq:evolution}, \eqref{eq:Neumann} and Lemma \ref{lem:time normal},
we deduce
\begin{equation*}
(\Delta R)_{\nu}=(d(\Delta R))(\nu)=(d(\partial_t R))(\nu)=(\partial_t (dR))(\nu)=-dR(\partial_t \nu)=-\frac{R}{2}R_{\nu}=0.
\end{equation*}
We complete the proof.
\end{proof}

We now divide $\mathcal{E}_{\partial}(t)$ into two parts.
We set
\begin{equation*}
\mathcal{N}_{\partial}(t):=\int_{M}\, R\log R \,dv,\quad \mathcal{R}_{\partial}(t):=\log \overline{R} \,\int_{M} R dv
\end{equation*}
such that $\mathcal{N}_{\partial}(t)=\mathcal{E}_{\partial}(t)+\mathcal{R}_{\partial}(t)$.
In view of Remark \ref{rem:relation},
$\mathcal{W}_{\partial}(t)$ can be expressed as
\begin{equation*}
\mathcal{W}_{\partial}(t)=\tau\left( \frac{d}{dt}\mathcal{N}_{\partial}(t)\right)-\mathcal{N}_{\partial}(t)-4\pi \chi(M)\log \tau;
\end{equation*}
in particular,
\begin{equation}\label{eq:preparation}
\frac{d}{dt}\mathcal{W}_{\partial}(t)=\tau\left(  \frac{d^2}{dt^2}\mathcal{N}_{\partial}(t)-\frac{2}{\tau}\frac{d}{dt}\mathcal{N}_{\partial}(t)+\frac{4\chi(M)}{\tau^2}   \right).
\end{equation}
Therefore,
it suffices to calculate the first two derivatives of $\mathcal{N}_{\partial}(t)$.
Thanks to Lemmas \ref{lem:useful} and \ref{lem:Laplacenorm},
we see the following (cf. \cite[Theorem 1.2 and Lemma 2.1]{Gu}):
\begin{lem}
\begin{align}\label{eq:1stderiv}
\frac{d}{dt}\mathcal{N}_{\partial}(t)&=\int_{M}\, \left((\Delta R) \log R+R^2\right)\,dv,\\ \label{eq:2ndtderiv}
\frac{d^2}{dt^2}\mathcal{N}_{\partial}(t)&=2\,\int_{M}\,R\, \left\Vert \frac{1}{2}Rg+\nabla^2 \log R \right\Vert^2 \,dv+2\int_{\partial M} \,\kappa\,R\Vert \nabla_{\partial M}\,(\log R)|_{\partial M} \Vert^2\, ds\\ \notag
&=2\int_{M}\,R \left\Vert \frac{1}{2}Rg+\nabla^2 \log R-\frac{1}{2\tau}g \right\Vert^2\,dv +\frac{2}{\tau} \frac{d}{dt}\mathcal{N}_{\partial}(t)-\,\frac{4\pi \chi(M)}{\tau^2}\\ \notag
&\quad +2\int_{\partial M} \kappa\,\left( R\Vert \nabla_{\partial M}\,(\log R)|_{\partial M} \Vert^2+\frac{1}{\tau^2}\right)\, ds.
\end{align}
\end{lem}
\begin{proof}
The first one \eqref{eq:1stderiv} follows from the same calculation as in the proof of \eqref{eq:time2.1}.
We consider the second one \eqref{eq:2ndtderiv}.
This has been obtained by \cite[Theorem 1.2 and Lemma 2.1]{Gu} for closed surfaces.
By the same calculation (without integration by parts),
one can verify
\begin{equation*}
\frac{d^2}{dt^2}\mathcal{N}_{\partial}(t)=\int_{M}\,\Delta (\Delta R+R^2)\log R+\frac{(\Delta R)^2}{R}+3R \Delta R+R^3\,dv.
\end{equation*}
In what follows,
taking care of the boundary terms,
we make use of integration by parts with \eqref{eq:Neumann}.
First,
we do it for the first term.
It holds that
\begin{align*}
\frac{d^2}{dt^2}\mathcal{N}_{\partial}(t)&=\int_{M}\, (\Delta R+R^2) \Delta \log R+\frac{(\Delta R)^2}{R}+3R \Delta R+R^3\,dv\\
&=\int_{M}\, (\Delta R+R^2) \left(  \frac{\Delta R}{R}-\frac{\Vert \nabla R\Vert^2}{R^2} \right)+\frac{(\Delta R)^2}{R}+3R \Delta R+R^3\,dv.
\end{align*}
Here we used Lemma \ref{lem:Laplacenorm}.
Further,
by applying the integration by parts to the term $\Vert \nabla R\Vert^2$,
\begin{equation*}
\frac{d^2}{dt^2}\mathcal{N}_{\partial}(t)=\int_{M}\, 2(\Delta R)(\Delta \log R)+\Delta R \Vert \nabla \log R\Vert^2+5R \Delta R+R^3\,dv.
\end{equation*}
We now apply integration by parts to the second term.
Then from Lemma \ref{lem:useful} we derive
\begin{align*}
\frac{d^2}{dt^2}\mathcal{N}_{\partial}(t)&=\int_{M}\, 2(\Delta R)(\Delta \log R)+R \Delta \Vert \nabla \log R\Vert^2+5R \Delta R+R^3\,dv\\
                                                             &\quad -\int_{\partial M}\,R (\Vert \nabla \log R\Vert^2)_{\nu}\,ds\\
                                                             &=\int_{M}\, 2(\Delta R)(\Delta \log R)+R \Delta \Vert \nabla \log R\Vert^2+5R \Delta R+R^3\,dv\\
                                                             &\quad +2\int_{\partial M} \kappa\, R\Vert \nabla_{\partial M}\,(\log R)|_{\partial M} \Vert^2\, ds.
\end{align*}
The rest is same as \cite{Gu},
and that is left to the reader.
\end{proof}

\begin{rem}
In view of \eqref{eq:2ndtderiv},
if $\bm$ is convex,
then $\mathcal{N}_{\partial}(t)$ is also convex.
\end{rem}

We are now in a position to prove Theorem \ref{thm:bdry Guo monotonicity}.
\begin{proof}[Proof of Theorem \ref{thm:bdry Guo monotonicity}]
Substituting \eqref{eq:2ndtderiv} into \eqref{eq:preparation},
we complete the proof.
\end{proof}

%
%

\subsection*{{\rm Acknowledgements}}
The authors thank Professor Jean C. Cortissoz for informing them of \cite{L}.
They are also grateful to the anonymous referees for their useful comments.
The first author was supported by JSPS KAKENHI (JP19K14521).
The second author was supported by JSPS Grant-in-Aid for Scientific Research on Innovative Areas ``Discrete Geometric Analysis for Materials Design" (17H06460).


\end{document}